\documentclass[10pt,oneside]{article}
\usepackage[OT4]{fontenc}
\usepackage{a4}
\usepackage{mathrsfs}
\usepackage{amsmath}
\usepackage{amssymb}
\usepackage{amsfonts}
\usepackage{graphicx}
\usepackage{color}
\newtheorem{theorem}{Theorem}
\newtheorem{lemma}[theorem]{Lemma}
\newtheorem{dfn}{Definition}
\newtheorem{obs}{Observation}

\newcommand{\bdf}{\begin{df} \begin{rm}}
\newcommand{\edf}{\end{rm} \end{df}}

\newenvironment{proof}{{\bf Proof.}}{\hspace*{\fill} \rule{2mm}{2mm} \par \hspace{0.1mm}}

\title{Graphs with $C_3$-free vertices are not universal fixers}
\author{ Magdalena Lema\'nska $^{1}$, 
 Monika Rosicka $^{1}$, Rita Zuazua $^{2}$
\\
\\
$^1${\small  Gda\'nsk University of Technology, Poland, }
\\ {\small   magda\@@mif.pg.gda.pl, rosamo\@@op.pl}
\\ 
$^2${\small  Science Faculty, UNAM, Mexico}
\\ {\small   ritazuazua\@@gmail.com}
\\ 
}
\date{}
\begin{document}

\maketitle

\begin{abstract}
A non-isolated vertex $x\in V(G)$ is called $C_{3}$-free if $x$ belongs to no
triangle of $G$. In \cite{BMW} Burger, Mynhardt and Weakley introduced the idea of universal fixers. Let $G=(V,E)$ be a graph with $n$ vertices and $G'$ a copy of $G$. For a bijective function $\pi:V(G)\mapsto V (G')$, we define the prism $\pi G$ of $G$ as follows: $V(\pi G)=V(G)\cup V(G')$ and $E(\pi G)=E(G)\cup E(G')\cup M_{\pi}$, where $M_{\pi}=\{u\pi (u): u\in V(G)\}$. Let $\gamma(G)$
be the domination number of $G$. If $\gamma(\pi G)=\gamma(G)$ for any bijective function $\pi$, then $G$ is called a universal fixer. In \cite{MX} it is conjectured that the only universal fixer is the edgeless graph $\overline{K_n}$. In this note, we prove that any graph $G$ with $C_3$-free vertices is not a universal fixer graph.

\vspace{.05cm}

\vspace{.05cm}

\end{abstract}
\section{Introduction}
Let $G=(V,E)$ be an undirected graph. The \textit{neighborhood} of a vertex $v\in V(G)$ in $G$ is the set $N_G(v)$ of all vertices adjacent to $v$ in $G$. For a set $X\subseteq V(G)$, the \textit{open neighborhood} $N_G(X)$ is defined to be $\bigcup_{v\in X} N_G(v)$, the \textit{closed neighborhood}, $N_G[X]=N_G(X)\cup X$ and we denote by $[X]$ the subgraph of $G$ induced by the set of vertices $X.$ For two sets of vertices $X,Y\subseteq V(G)$, we denote by $E(X,Y)$ the set of edges $xy\in E(G)$ such that $x\in X$ and $y\in Y.$

A set $D\subseteq V(G)$ is a \textit{dominating set} of $G$ if $N_G[D]=V.$ The \textit{domination number} of $G$, denoted by $\gamma(G)$, is the minimum cardinality of a dominating set in $G.$ A $\gamma$-set of $G$ is a dominating set of $G$ of cardinality $\gamma(G).$ If a set of vertices $A\subseteq V(G)$ is dominated by a set of vertices $D\subseteq V(G)$, that is, if for every vertex $v\in A$ there is $u\in D$ such that $vu\in E(G)$, we write $D\succ A.$ A set $S\subset V(G)$ is a \textit{2-packing} of $G$ if $N_G[u]\cap N_G[v]=\emptyset$ for every distinct $u, v\in S.$

\begin{dfn}
Let $G=(V,E)$ be a graph and $G'$ a copy of $G$. For a bijective function $\pi:V(G)\mapsto V(G'),$ we define the prism $\pi G$ of $G$ as follows:

$$V(\pi G)=V(G)\cup V(G')\ \ and \ \ E(\pi G)=E(G)\cup E(G')\cup M_{\pi},$$ 

where $M_{\pi}=\{u\pi(u): u\in V(G)\}.$
\end{dfn}

It is clear that every permutation $\pi$ of $V(G)$ defines a bijective function from $V(G)$ to $V(G')$, so we indistinctly use the permutation $\pi$ of $V(G)$ or the associated bijection $\pi: V(G)\mapsto V(G').$

It is known \cite{MX} that for any permutation $\pi $ and any graph $G$, the domination numbers $\gamma (G)$ and $\gamma (\pi G)$ have the following relation: $\gamma (G)\leq \gamma (\pi G)\leq 2\gamma (G)$. The graph $G$ is called a universal fixer if $\gamma(\pi G)=\gamma(G)$ for every permutation $\pi$ of $V(G).$

Universal fixers were studied in \cite{MX} for several classes of graphs and it was conjectured that the edgeless graph $\overline{K_n}$ is the only universal fixer. In \cite{BM}, \cite{CGM} and \cite{RGG} it is shown that regular graphs, claw-free graphs and bipartite graphs are not universal fixers.

\section{Useful lemmas}
In what follows, we suppose that the graph $G=(V,E)$ has $n$ vertices. For $x\in V(G)$, the copy of $x$ in $V(G')$ is denoted by $x'.$ Similarly, if $A\subset V(G),$ we denote $A'=\{u'\in V(G'): u\in A\}.$

\begin{dfn}
A $\gamma$-set $A$ of $G$ is called a separable $\gamma$-set (or an $A_1$-$\gamma$-set) if it can be partitioned into two nonempty subsets $A_1$ and $A_2,$ such that $A_1\succ V-A.$
\end{dfn}

In \cite{MX}, the following lemma is proved:

\begin{lemma}
If $A=A_1\cup A_2$ is an $A_1$-$\gamma$-set of $G,$ then: 
\begin{enumerate}
\item $A_2$ is an independent set;
\item $E(A_1,A_2)=\emptyset$;
\item $A_2$ is a 2-packing of $G$.
\end{enumerate}
\end{lemma}

For an $A_1$-$\gamma$-set $A$ in $G$ and a permutation $\pi$, we  denote $B_1^{'}=\pi(A_1), B_2^{'}=\pi(A_2)$ and $B^{'}=\pi(A)$. It is clear that $B^{'}=B_1^{'}\cup B_2^{'}.$ If the permutation $\pi$ is the identity, then $B_1^{'}=A_1^{'}$ and $B_2^{'}=A_2^{'}.$

\begin{dfn}
We say that an $A_1$-$\gamma$-set $A=A_1\cup A_2$ is effective under some permutation $\pi$ if $B^{'}=\pi(A)$ is a $B_2^{'}$-$\gamma$-set in $G'$, where $B_2^{'}=\pi(A_2).$
\end{dfn}

\begin{obs}
By Lemma 1, if an $A_1$-$\gamma$-set $A$ is effective under some permutation $\pi$, then $B_1^{'}=\pi(A_1)$ is an independent set, $E(B_1^{'},B_2^{'})=\emptyset$ and $B_1^{'}$ is a 2-packing in $G'.$
\end{obs}

The next Theorem is proved in \cite{MX}.

\begin{theorem}
A graph $G\neq\overline{K_n}$ is a universal fixer if for every permutation $\pi$ of $V(G)$ there exists a separable $\gamma$-set which is effective under $\pi$.
\end{theorem}

\section{Main result}
\begin{dfn} Let $G$ be a graph. We say that a non-isolated vertex $x$ is $C_{3}$-free if $x$ belongs to no
triangle of $G$. 
\end{dfn}

Observe that $x\in V(G)$ is $C_3$-free if and only if $[N_G[x]]\cong K_{1,m}$ with $m\geq 1.$ Our main result is the following theorem.

\begin{theorem}
If $G$ has a $C_3$-free vertex $x,$  then $G$ is not a universal fixer.
\end{theorem}
\begin{proof}
Let $G$ be a graph and $x$ a $C_{3}$-free vertex of $G$. Denote by $X$ the
subgraph $\left[ N_{G}\left[ x\right] \right] \cong K_{1,m}$ of $G$ for 
$m\geq 1$. We define $\pi :V(G)\rightarrow V(G^{\prime })$ such that:

\begin{enumerate}
\item $\pi (v)=v^{\prime }$ for every $v\in V(G)-X$,

\item $\pi (v)\neq v^{\prime }$ for every $v\in X$ and 

\item if $m\geq 2$ and $\pi (v)=u^{\prime },$ then $\pi (u)\neq v^{\prime }$
for every $u,v\in X$.
\end{enumerate}

Let $A=A_{1}\cup A_{2}$ be an arbitrary separable $\gamma $-set of $G.$ We
prove that $A$ is not effective under $\pi .$ 

Since $X=\left[ N_{G}\left[ x\right] \right] ,$ we get that $\left\vert
A\cap X\right\vert \geq 1$ for every dominating set $A$. We consider the
following three cases.\\

\textit{Case 1.} $A\cap X=\left\{ v\right\} $.

\begin{enumerate}
\item[(1)] If $v\in A_{1},$ then $\pi (A_{2})=A_{2}^{\prime }=B_{2}^{\prime }
$ and $v^{\prime }\in V(G^{\prime })-B^{\prime }$.\textit{\ }By Lemma 1, we
have that $E(A_{1},A_{2})=\emptyset $. Thus $E(\left\{ v^{\prime }\right\}
,A_{2}^{\prime })=\emptyset$ and it is a contradiction to the fact that $%
B_{2}^{\prime }=A_{2}^{\prime }\succ V(G^{\prime })-B^{\prime }$.

\item[(2)] If $v\in A_{2}$ and $\pi (v)=w^{\prime }$, then the definition of 
$A_{1}$ implies that there exists $u\in A_{1}$ such that $uw\in E(G)$. Hence 
$u^{\prime }w^{\prime }\in E(G^{\prime })$ and consequently, $E(\pi
(A_{1}),\pi (A_{2}))\neq \emptyset,$ which is a contradiction to
Observation 1.
\end{enumerate}

\textit{Case 2. }$A\cap X=\left\{ u,v\right\} $. Recall that for every $%
u,v\in X$ we have that either $uv\in E(G),$ or $uv\notin E(G)$ and $N(u)\cap
N(v)\neq \emptyset$.

\begin{enumerate}
\item[(1)] If $u,v\in A_{1}$, then $\left\vert B_{1}^{\prime }\cap X^{\prime
}\right\vert =2$. Thus $B_{1}^{\prime }$ is neither independent nor a $2$%
-packing, a contradiction to Observation 1.

\item[(2)] If $u,v\in A_{2}$, then $A_{2}$ is neither independent nor a $2$%
-packing, a contradiction to Lemma 1.
\end{enumerate}

From (1) and (2), we conclude that 
\begin{equation}
\left\vert A_{1}\cap X\right\vert \leq 1\text{ and }\left\vert A_{2}\cap
X\right\vert \leq 1  \label{case2}
\end{equation}
for every $A_{1}$-$\gamma $-set and therefore we only have to consider the
case when $u\in A_{1}$ and $v\in A_{2}$. Since $E(A_{1},A_{2})=\emptyset$, 
vertices $u$, $v$ and $x$ are all distinct. Analogously, since 
$E(B_{1}^{\prime },B_{2}^{\prime })=\emptyset,$ vertices $\pi (u)$, $\pi
(v)$ and $x^{\prime }$ are all distinct too. By the definition of $\pi $, if 
$\pi (y)=z^{\prime },$ then $\pi (z)\neq y^{\prime }$ for every $y\in X$.
Thus there exists $z\in X\cap (V(G)-A)$ such that either $\pi (u)=z^{\prime
}\neq v^{\prime }$ or $\pi (v)=z^{\prime }\neq u^{\prime }$. 

\begin{enumerate}
\item Suppose $\pi(u)=z'\neq v'$.  Since $A_1\succ V-A$, there exists $w\in A_1$ such that $zw\in E(G)$ and $\pi (w)=w'\in B_1'$. Then $z'w'\in E(B_1',B_1'),$ which is a contradiction with Observation 1.
\item Suppose that $\pi(v)=z'\neq u'$.  Since $A_1\succ V-A$, there exists $w\in A_1$ such that $zw\in E(G)$ and $\pi (w)=w'\in B_1'$. Then $z'w'\in E(B_1', B_2'),$ which is a contradiction with Observation 1.
\end{enumerate}

\textit{Case 3. }$\left\vert A\cap X\right\vert \geq 3.$ In this case, $
\left\vert A_{1}\cap X\right\vert \geq 2$ or $\left\vert A_{2}\cap
X\right\vert \geq 2$, which is impossible by (\ref{case2}) of Case 2.\\

From these cases, we conclude that every separable $\gamma $-set is not
effective under $\pi $. By Theorem $2$, the graph $G$ is not a universal fixer.
\end{proof}

Recall that the girth of a graph $G$ is the length of a
shortest cycle contained in the graph. Notice that in a graph $G$ with girth
four or more, every vertex is $C_{3}$-free. In particular, the bipartite graphs satisfy this condition.\\

\textbf{Acknowledgements}

We thank Bernardo Llano for useful comments. The authors thank the financial support received from
Grant UNAM-PAPIIT IN-117812 and SEP-CONACyT.

\end{document}